\documentclass[12pt, reqno]{amsart}
\usepackage[PS]{diagrams}

\usepackage{pgf}% need this for hyperref below - not sure why!
\usepackage[hypertexnames=false,colorlinks=false,pdfborderstyle={/S/U/W 1}]{hyperref}

\numberwithin{equation}{section}

\newtheorem{theorem}[equation]{Theorem}
\newtheorem{proposition}[equation]{Proposition}
\newtheorem{corollary}[equation]{Corollary}
\newtheorem{lemma}[equation]{Lemma}

\theoremstyle{definition}
\newtheorem{definition}[equation]{Definition}

\newtheorem{notation}[equation]{Notation}

\newcommand{\id}{\mathrm{id}}

\newcommand{\powers}[1]{[\kern -1pt [{#1}] \kern -1pt ]} %Tighter [[-]]
\newcommand{\nov}[1]{(\kern -1.7pt ( {#1})\kern-1.7pt )} %Tighter ((-))
\newcommand{\tensor}{\mathop{\otimes}}
\newcommand{\bZ}{\mathbb{Z}}

\newcommand{\oo}{\mathcal{O}}
\newcommand{\tprodrt}{\ensuremath{\sideset{}{^\mathrm{rt}}{\mathop{\underset{\widetilde\ }{\prod}}}}}
\newcommand{\tprodlt}{\ensuremath{\sideset{^\mathrm{lt}}{}{\mathop{\underset{\widetilde\ }{\prod}}}}}

\newcommand{\T}{\mathfrak{T}} % algebraic torus
\newcommand{\cone}{\mathrm{cone}}
\newcommand{\tot}{\mathrm{Tot}}
\newcommand{\totrt}{\ensuremath{\mathrm{Tot}^{\mathrm{rt}}\,}}

\newcommand{\hh}[1]{\ensuremath{\mathbb{#1}}}

\renewcommand{\phi}{\varphi}

\begin{document}

\author{Thomas H\"uttemann}
\author{Luke Steers}

\email{t.huettemann@qub.ac.uk}

\urladdr{www.qub.ac.uk/puremaths/Staff/Thomas Huettemann/}

\address{Pure Mathematics Research Centre\\
School of Mathematics and Physics\\
Queen's University Belfast\\
Belfast BT7 1NN\\
Northern Ireland, UK}

\title[Finite domination over strongly $\bZ$-graded rings]%
{Finite domination and Novikov homology over strongly $\bZ$-graded rings}

\date{\today}


\begin{abstract}
  Let $L$ be a unital $\bZ$-graded ring, and let $C$ be a bounded
  chain complex of finitely generated $L$-modules. We give a
  homological characterisation of when $C$ is homotopy equivalent to a
  bounded complex of finitely generated projective
  $L_{0}$\nobreakdash-modules, generalising known results for twisted
  \textsc{Laurent} polynomial rings. The crucial hypothesis is that
  $L$ is a {\it strongly\/} graded ring.
\end{abstract}

\maketitle

\section{Finite domination over strongly $\bZ$-graded rings}

\subsection*{Finite domination and Novikov homology}

Let $L$ be a unital ring, and let $K$ be a subring of~$L$. A bounded
chain complex~$C$ of (right) $L$-modules is {\it $K$-finitely
  dominated\/} if $C$, considered as a complex of
$K$\nobreakdash-modules, is a retract up to homotopy of a bounded
complex of finitely generated free $K$-modules; this happens if and
only if $C$ is homotopy equivalent, as a $K$-module complex, to a
bounded complex of finitely generated projective $K$-modules
\cite[Proposition~3.2.~(ii)]{RATFO}. The following result of
\textsc{Ranicki} gives a complete homological characterisation of
finite domination in an important special case:

\begin{theorem}[\textsc{Ranicki \cite[Theorem~2]{RFDNR}}]
  \label{thm:orinigal}
  Let $R$ be a unital ring, and let $R[t,t^{-1}]$ denote the
  \textsc{Laurent} polynomial ring in the indeterminate~$t$. Let $C$
  be a bounded chain complex of finitely generated free
  $R[t,t^{-1}]$-modules. The complex~$C$ is $R$-finitely dominated if
  and only if both
  \begin{displaymath}
    C \tensor_{R[t,t^{-1}]} R\nov{t^{-1}} \qquad \text{and} \qquad C
    \tensor_{R[t,t^{-1}]} R\nov{t}
  \end{displaymath}
  have vanishing homology in all degrees. Here $R\nov{t} =
  R\powers{t}[t^{-1}]$ is the ring of formal \textsc{Laurent} series
  in~$t$, and similarly $R\nov{t^{-1}} = R\powers{t^{-1}}[t]$ stands
  for the ring of formal \textsc{Laurent} series in~$t^{-1}$.
\end{theorem}

The cited paper \cite{RFDNR} also contains a discussion of the
relevance of finite domination in topology. --- The rings~$R\nov{t}$
and~$R\nov{t^{-1}}$ are known as \textsc{Novikov} rings. The theorem
can be formulated more succinctly: {\it The chain complex~$C$ is
  $R$\nobreakdash-finitely dominated if and only if it has trivial
  \textsc{Novikov} homology}.

\medbreak

In the present paper we formulate and prove a surprising
strengthening: {\it Theorem~\ref{thm:orinigal} remains valid if
  $R[t,t^{-1}]$ is replaced by an arbitrary strongly $\bZ$-graded
  ring}, provided the definition of \textsc{Novikov} rings is adapted
suitably. We start by recalling the requisite definitions.

\begin{definition}
  A \emph{\(\bZ\)-graded ring} is a (unital) ring \(L\) equipped with
  a direct sum decomposition into additive subgroups
  \(L=\bigoplus_{k\in\bZ} L_k\) such that \(L_kL_\ell\subseteq
  L_{k+\ell}\) for all \( k,\ell\in\bZ\), where \(L_kL_\ell\) consists
  of the finite sums of ring products \(xy\) with \(x \in L_k\) and
  \(y \in L_\ell\). The summands~$L_{k}$ are called the {\it
    (homogeneous) components\/} of~$L$; elements of~$L_{k}$ are called
  {\it homogeneous of degree~$k$}. --- Following \textsc{Dade}
  \cite{GRD} we call $L$ a {\it strongly $\bZ$-graded ring} if $L_{k}
  L_{\ell} = L_{k+\ell}$ for all $k, \ell \in \bZ$.
\end{definition}

A specific example of a strongly $\bZ$-graded ring is $L =
R[t,t^{-1}]$, the ring of \textsc{Laurent} polynomials; the $n$th
component is $\{rt^{n}\,|\, r \in R\}$. The reader may wish to keep
this motivating example in mind.

\medbreak

We will use the symbol $R[t,t^{-1}] = \bigoplus_{k \in \bZ} R_{k}$ for
an arbitrary $\bZ$\nobreakdash-graded ring in this paper. One may
think of the elements of~$R[t,t^{-1}]$ as formal \textsc{Laurent}
polynomials $\sum_{j=m}^{n} a_{j} t^{j}$ with $a_{j} \in R_{j}$, but
note that this is a purely notational device; in general the ring
$R[t,t^{-1}]$ does not contain an element called~$t$.  The point of
using this notation is that we have a rather suggestive way of
denoting various rings and modules constructed from~$R[t,t^{-1}]$. For
example, we can introduce the \textsc{Novikov} rings
\begin{displaymath}
  R\nov{t^{-1}} = \prod_{n \leq 0} R_{n} \,\oplus\, \bigoplus_{n > 0}
  R_{n}
  \quad \text{and} \quad
  R\nov{t} = \bigoplus_{n < 0} R_{n} \,\oplus\, \prod_{n \geq 0} R_{n}
\end{displaymath}
and think of their elements as formal power series $\sum_{j \in \bZ}
a_{j} t^{j}$ with $a_{j} \in R_{j}$ such that $a_{j} = 0$ whenever $j
\gg 0$ or $j \ll 0$, respectively.

\medbreak

It is known that in any $\bZ$-graded ring the unit element is
necessarily homogeneous of degree~$0$ so that $R_{0}$ is a subring
of~$R[t,t^{-1}]$. With these preliminaries in place we can formulate
our main result:

\begin{theorem}
  \label{thm:main}
  Let $R[t,t^{-1}] = \bigoplus_{k \in \bZ} R_{k}$ be a strongly
  $\bZ$-graded ring, and let $C$ be a bounded chain complex of
  finitely generated free $R[t,t^{-1}]$-modules. The complex~$C$ is
  $R_{0}$-finitely dominated if and only if both
  \begin{displaymath}
    \label{eq:cond}
    C \tensor_{R[t,t^{-1}]} R\nov{t^{-1}} \qquad \text{and} \qquad C
    \tensor_{R[t,t^{-1}]} R\nov{t}
  \end{displaymath}
  have vanishing homology in all degrees.
\end{theorem}

As a special case this says that Theorem~\ref{thm:orinigal} holds for
twisted \textsc{Laurent} polynomial rings \cite[Theorem~6]{MR2282258},
or even for the more general case of crossed products (which are
characterised by having homogeneous invertible elements of arbitrary
degrees). However, this is not the complete extent of the
generalisation as there are strongly graded rings which are not
crossed products. Possibly the easiest example to write down is the
following: Let $K[A,B,C,D]$ be a polynomial ring over the field~$K$,
considered as a $\bZ$-graded ring by giving $A$ and~$C$ degree~$1$,
and giving $B$ and~$D$ degree~$-1$. Let $R[t,t^{-1}]$ be the quotient
$K[A,B,C,D]/ (AB+CD-1)$; as the relation is homogeneous, this results
in a $\bZ$-graded ring which is actually strongly graded since
$AB+CD=BA+DC=1$ by construction. It can be shown, using ideas from
\textsc{Gr\"obner} basis theory, that the only units are $K^{\times}
\subset R[t,t^{-1}]$ so that our ring is not a crossed product. Now
consider the following 2-step chain complex:
\begin{displaymath}
  R[t,t^{-1}] \rTo[l>=3em]^{\left(1-A \atop 1-B \right)} R[t,t^{-1}]
  \oplus R[t,t^{-1}] \rTo[l>=6em]^{\left( 1-B, -(1-A) \right)}
  R[t,t^{-1}] 
\end{displaymath}
The map $1-A$ becomes invertible in~$R\nov{t}$, with inverse
$(1-A)^{-1} = \sum_{j \geq 0} A^{j}$; similarly, the map $1-B$ becomes
invertible in~$R\nov{t^{-1}}$. Hence the complex becomes acyclic after
tensoring with $R\nov{t^{\pm 1}}$, and Theorem~\ref{thm:main} asserts
that it is in fact $R_{0}$-finitely dominated.

\subsection*{Structure of the paper}

For the proof of Theorem~\ref{thm:main} we combine techniques from
strongly graded algebra with homotopy-theoretic methods and
homological algebra of bicomplexes.  We start by introducing various
rings associated to a $\bZ$-graded ring, and discuss partitions of
unity which are the main technical tool from graded algebra to be used
throughout the paper. This will occupy the remainder of~\S1.

In \S2 we prove the ``if'' implication of Theorem~\ref{thm:main},
based on the homotopy-theoretic methods used in \cite{RFDNR} for the
case of a \textsc{Laurent} polynomial ring. The organisation follows
the pattern laid out by the first author in~\cite{TVB}, where a
description of the algebro-geometric background of the procedure is
given. It is of interest to note that the $\bZ$-graded structure of
our ring allows us in Proposition~\ref{exchcos} to construct complexes
of sheaves from the given complex of modules~$C$, while the {\it
  strong\/} grading ensures that certain chain complexes consist of
finitely generated projective $R_{0}$-modules,
cf.~Corollary~\ref{cor:H_of_O}.

In \S3 we attack the reverse implication of Theorem~\ref{thm:main},
using double complex techniques as documented in~\cite{DCVNC}. The
graded structure is used at various places. Most notably, the
definition and the properties of the ``algebraic torus'', a substitute
for the more usual algebraic mapping torus of a self-map of a chain
complex, depend crucially on extra data which can be chosen only in
view of the strong grading. In addition, passage to \textsc{Novikov}
rings involves a certain ``twisting'' operation on powers of modules
that is defined in terms of the grading.

The results of this paper were obtained as part of the second author's
PhD thesis.

\subsection*{Rings associated with $\bZ$-graded rings}

We make the following conventions for the rest of the paper: All
rings, graded or otherwise, are assumed unital and all modules right
unless stated otherwise. We let $R[t,t^{-1}]$ stand for an arbitrary
unital $\bZ$-graded ring, with $n$th homogeneous component denoted
by~$R_{n}$. That is, we have a graded ring $R[t,t^{-1}] = \bigoplus_{n
  \in \bZ} R_{n}$. In many cases we will assume this ring to be
strongly graded, but will take care to indicate where this hypothesis
is really needed.

\medbreak

Given a $\bZ$-graded ring~$R[t,t^{-1}]$, it is known that the unit
element~$1$ must be homogeneous of degree~$0$
\cite[Proposition~1.4]{GRD}. It is then immediate from the definition
that $R_{0}$~is a subring of~$R[t,t^{-1}]$, and that all the
homogeneous components~$R_{k}$ are $R_{0}$-bimodules.

\medbreak

Given the \(\bZ\)-graded ring \(R[t,t^{-1}]\) we define two
$\bZ$-graded subrings by setting $R[t^{-1}]=\bigoplus_{k \leq 0} R_k$
and $R[t] = \bigoplus_{k \geq 0} R_{k}$. These graded rings have
trivial components in all positive and negative degrees,
respectively. Elements can be thought of as formal polynomials in
$t^{-1}$ and~$t$, respectively, with the coefficient of~$t^{j}$ an
element of~$R_{j}$.

We can also define the analogues of power series rings,
$R\powers{t^{-1}} = \prod_{n \leq 0} R_{n}$ and $R\powers{t} =
\prod_{n \geq 0} R_{n}$. Elements are of course formal power series
in $t^{-1}$ and~$t$, respectively, with coefficient of~$t^{j}$ an
element of~$R_{j}$. We have previously defined the \textsc{Novikov}
rings $R\nov{t^{-1}}$ and $R\nov{t}$. Note that power series and
\textsc{Novikov} rings are {\it not\/} considered as graded rings; in
fact, they do not admit a natural $\bZ$-grading.

The collection of rings fits into the commutative diagram of
ring inclusions displayed in Fig.~\ref{fig:rings}.

\begin{figure}[ht]
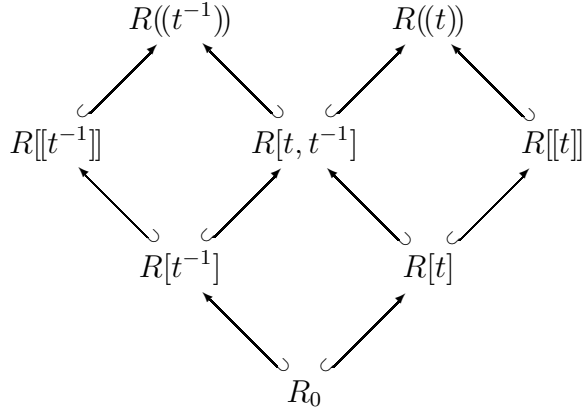

  \centering
  \begin{diagram}[small]
    && R\nov{t^{-1}} & &&& R\nov{t} && \\
    & \ruInto  && \luInto && \ruInto && \luInto \\
    R\powers{t^{-1}} &&&& R[t,t^{-1}]&&&& R\powers{t} \\
    &\luInto && \ruInto && \luInto && \ruInto\\
    && R[t^{-1}] &&&& R[t] &&  \\
    &&& \luInto && \ruInto &&&\\
    &&&& R_0 &&&&
  \end{diagram}
  \caption{The collection of rings and their inclusion relation}
  \label{fig:rings}
\end{figure}

\subsection*{Partitions of unity and strongly graded rings}

\begin{definition}
  Given \(n\in\bZ\), an expression of the form \(1=\sum_{j} u_j v_j\)
  with \(u_j\in R_n,\,v_j\in R_{-n}\) is called a \emph{partition of
    unity of type \((n,-n)\)}. This is understood to be a finite sum;
  we do not specify the summation range unless we need it explicitly.
\end{definition}

A partition of unity of type~$(n,-n)$ exists if and only if $1 \in
R_{n} R_{-n}$. Partitions of unity are our main technical tool; their
existence is intimately related to the graded structure of the ring:

\begin{proposition}[Characterisation of strongly graded rings]
  \label{prop:characterisation_strongly_graded}
  The following statements are equivalent:
  \begin{enumerate}
  \item \label{item:1} The ring \(R[t,t^{-1}]\) is strongly graded.
  \item \label{item:2} For every \(n\in\bZ\) there is at least one
    partition of unity of type \((n,-n)\).
  \item \label{item:3} There is at least one partition of unity of
    type~\((1,-1)\), and at least one of type~\((-1,1)\).
  \end{enumerate}
\end{proposition}

\begin{proof}
  For the equivalence of statements~\eqref{item:1} and \eqref{item:2}
  see Proposition~1.6 of \cite{GRD}. That \eqref{item:2}
  implies~\eqref{item:3} is trivial.  For the converse, suppose that
  $1 = \sum_{j=1}^{q} u_{j} v_{j}$ is a partition of unity of type
  $(\pm 1, \mp 1)$; then for $m \geq 2$ the $q^{m}$ pairs of elements
  \begin{displaymath}
    u_{\mathbf{j}} = u_{j_{1}} u_{j_{2}} \cdots u_{j_{m}} \quad
    \text{and} \quad v_{\mathbf{j}} = v_{j_{m}} v_{j_{m-1}} \cdots
    v_{j_{1}} \ ,
  \end{displaymath}
  where $\mathbf{j} = (j_{1}, j_{2}, \cdots, j_{m}) \in \{1, 2,
  \cdots, q\}^{m}$, form a partition of unity of type $(\pm m, \mp
  m)$.
\end{proof}

The following Proposition is well known; we include a proof because of
its fundamental importance for this paper.

\begin{proposition}\label{fgpro}
  If \(R[t,t^{-1}]\) is strongly graded, then each of the \(R_k\) is
  finitely generated projective both as a left \(R_0\)-module and as a
  right \(R_0\)-module.
\end{proposition}

\begin{proof}
  We treat the case of right \(R_0\)-modules only. Let \(1=\sum_{j}
  u_jv_j\) be a partition of unity of type~$(k,-k)$; existence is
  guaranteed by
  Proposition~\ref{prop:characterisation_strongly_graded}. The maps
  \(g_j \colon R_{k} \rightarrow R_{0}\) with \(g_j (y) = v_j y\) are maps
  of right \(R_0\)-modules and satisfy \(\sum_{j} u_jg_j(r) = \sum_{j}
  u_{j} v_{j} r = r\) for any $r \in R_{k}$. Thus \(R_k\) is a
  finitely generated projective right \(R_0\)-module by the dual basis
  lemma, cf.~Proposition~12 of \cite[\S{}II.2.7]{MR1727844}.
\end{proof}

\begin{corollary}
  \label{cor:stays_projective}
  Suppose $R[t,t^{-1}]$ is strongly graded. Then any projective left
  or right $R[t,t^{-1}]$-module is also projective when considered as
  a left or right $R_{0}$\nobreakdash-module.\qed
\end{corollary}

Given numbers \(q,p\in\bZ\) we define the symbols
\begin{displaymath}
  t^{q}\cdot R[t^{-1}] = \bigoplus_{j\leq q} R_j \quad \text{and}
  \quad t^{-p}\cdot R[t]  = \bigoplus_{j\geq -p} R_j \ ;
\end{displaymath}
the former is an $R[t^{-1}]$-bimodule, the latter an
$R[t]$-bimodule. The induced $R[t,t^{-1}]$-modules behave as expected
in the strongly graded case:

\begin{lemma}
  \label{lem:sfcdn}
  Let \(q,p\in\bZ\). The $R[t,t^{-1}]$-linear maps
  \begin{gather*}
    \gamma \colon t^{q}\cdot R[t^{-1}] \tensor_{R[t^{-1}]} R[t,t^{-1}]
    \rTo R[t,t^{-1}] \ , \quad r\tensor s\mapsto rs \\%
    \intertext{and}%
    \alpha \colon t^{-p}\cdot R[t] \tensor_{R[t]} R[t,t^{-1}] \rTo
    R[t,t^{-1}] \ , \quad r \tensor s\mapsto rs
  \end{gather*}
  are isomorphisms provided the ring \(R[t,t^{-1}]\) is strongly
  graded.
\end{lemma}

\begin{proof}
  Suppose \(R[t,t^{-1}]\) is strongly graded. Then we may choose a
  partition of unity of type \((-p,p)\), say \(1= \sum_{j} u_jv_j\)
  with \(u_j\in R_{-p}\) and \(v_j\in R_{p}\).  The
  $R[t,t^{-1}]$-linear map
  \begin{displaymath}
    \beta \colon R[t,t^{-1}] \rTo t^{-p}\cdot R[t]
    \tensor_{R[t]} R[t,t^{-1}] \ ,\quad r \mapsto \sum_{j} u_{j}
    \tensor v_{j} r
  \end{displaymath}
  satisfies \(\alpha\beta(r)= \sum_{j} u_{j} v_{j} r = r\) so
  that $\alpha\beta = \id$. Also
  \begin{displaymath}
    \beta\alpha (r\tensor s) = \beta(rs) = \sum_{j} u_{j}
    \tensor v_{j} rs \underset{(*)}= \sum_{j} u_{j} v_{j} r
    \tensor s = r \tensor s
  \end{displaymath}
  (where the equality labelled $(*)$ is true since $v_{j} r \in R[t]$
  for $r \in t^{-p} \cdot R[t]$), whence \(\beta \alpha = \id\). ---
  The case of \(\gamma\) is similar.
\end{proof}

\section{Trivial Novikov homology implies finite domination}

\subsection*{Sheaves and their cohomology}

We will have occasion to study diagrams of the form
\begin{equation}
  \label{eq:diagram}
  \mathfrak{M} = \ \Big( M^{-} \rTo^{\mu^{-}} M \lTo^{\mu^{+}} M^{+} \Big)
  \ ;
\end{equation}
the entries will be modules, or chain complexes of modules. The
maps~$\mu^{-}$ and~$\mu^{-}$ are called the {\it structure maps
  of~$\mathfrak{M}$}. A {\it map of diagrams\/} consists of a triple
of maps~$(f^{-}, f, f^{+})$ which is compatible with the structure
maps of source and target.

\begin{definition}
  \label{shfcond}
  Let as before $R[t,t^{-1}]$ be a $\bZ$-graded ring. A {\it
    pre-sheaf\/} is a diagram~$\mathfrak{M}$ of the
  form~\eqref{eq:diagram} where $M^{-}$ is an $R[t^{-1}]$-module, $M$
  is an $R[t,t^{-1}]$-module, $M^{+}$ is an $R[t]$-module, $f^{-}$ is
  $R[t^{-1}]$-linear and $f^{+}$ is $R[t]$-linear. The pre-sheaf
  $\mathfrak{M}$ is called a {\it sheaf\/} if the adjoints of the
  structure maps~$f^{-}$ and~$f^{+}$ are isomorphisms of
  $R[t,t^{-1}]$-modules:
  \begin{displaymath}
    M^{-} \tensor_{R[t^{-1}]} R[t,t^{-1}] \rTo^{\cong} M \lTo^{\cong}
    M^{+} \tensor_{R[t]} R[t,t^{-1}]
  \end{displaymath}
\end{definition}

Of particular importance will be the pre-sheaves
\begin{equation}
  \label{eq:Oqp}
  \oo (q ,p) = \ \Big(t^q\cdot R[t^{-1}]
  \rTo^{\subset}_{\iota_{q}} R[t,t^{-1}] \lTo^{\supset}_{{}_{p}\iota}
  t^{-p}\cdot R[t]\Big)
\end{equation}
which depend on the numbers $q,p \in \bZ$. In case $R[t,t^{-1}]$ is
strongly graded these pre-sheaves are actually sheaves by
Lemma~\ref{lem:sfcdn}, and are then called {\it twisting sheaves}.

\medbreak

Back to a general diagram~$\mathfrak{M}$ of modules of the
form~\eqref{eq:diagram}, we define:

\begin{definition}
  The $R_{0}$-module chain complex
  \begin{displaymath}
    H(\mathfrak{M}) = \ \Big( M \lTo^{-f^{-} + f^{+}} M^- \oplus M^+
    \Big)
  \end{displaymath}
  (concentrated in chain degrees \(-1\) and \(0\)) is called the
  \emph{cohomology chain complex} of~$\mathfrak{M}$. We write
  $H^{q}(\mathfrak{M})$ for the $(-q)$th homology
  of~$H(\mathfrak{M})$.
\end{definition}

In fact, $H^{q}(\mathfrak{M}) = \lim^{q} (\mathfrak{M})$. --- The
definitions of pre-sheaf and sheaf apply to chain complexes instead of
modules {\it mutatis mutandis}; in effect, a (pre-)sheaf of chain
complexes is the same as a chain complex of (pre-)sheaves. Given any
diagram of chain complexes \(\mathfrak{N}= \big(N^- \rTo^{g^{-}} N
\lTo^{g^{+}} N^+ \big)\) we obtain a double complex
\(H(\mathfrak{N})\)
by applying the cohomology chain complex construction levelwise. (The
double complex is concentrated in columns~$-1$ and~$0$, and has
commuting differentials.)

\begin{definition}
  Given a diagram of chain complexes \(\mathfrak{N}\) we define its
  {\it hypercohomology complex\/} $\hh{H}(\mathfrak{N})$ by
  setting $\hh{H}(\mathfrak{N}) = \mathrm{Tot} H(\mathfrak{N})$,
  the totalisation of~$H(\mathfrak{N})$.
\end{definition}

\noindent The totalisation is the usual one:
$\hh{H}(\mathfrak{N})_{n} = N^{-}_{n} \oplus N^{+}_{n} \oplus
N_{n+1}$, with differential induced by $-g^{-}$, $g^{+}$, the
differentials of~$N^{-}$ and~$N^{+}$, and the negative of the
differential of~$N$. Up to shift and sign conventions
$\hh{H}(\mathfrak{N})$ is the mapping cone of the map
\(-g^-+g^+\).

\begin{proposition}\label{hogrs}
  Let \(R[t,t^{-1}]\) be a $\bZ$-graded ring, and let $q,p \in \bZ$.
  \begin{enumerate}
  \item For \(p+q\geq 0\), the complex \(H\big(\oo(q,p)\big)\) is
    homotopy equivalent to the chain complex having
    \(\bigoplus_{k=-p}^{q} R_k\) in chain level \(0\) as its only
    non-trivial chain module.
  \item For \(p+q=-1\), the complex \(H\big(\oo(q,p)\big)\) is
    contractible.
  \item For \(p+q\leq-2\), the complex \(H\big(\oo(q,p)\big)\) is
    homotopy equivalent to the chain complex having
    \(\bigoplus_{k=q+1}^{-p-1} R_k\) in chain level \(-1\) as its only
    non-trivial chain module.
  \end{enumerate}
\end{proposition}

\begin{proof}
  We consider the case \(p+q\geq 0\) only, the others being similar
  (and quite irrelevant for our purposes). It is enough to show that
  the $R_{0}$-module sequence
  \begin{displaymath}
    0 \rTo \bigoplus_{k=-p}^{q} R_k %
    \pile {\rTo^{\Delta} \\ \lDashto_{\rho}} %
    t^{q}\cdot R[t^{-1}] \oplus t^{-p}\cdot R[t] %
    \pile{\rTo[l>=4em]^{- \iota_q+ {}_p\iota} \\ \lDashto_{\sigma}} %
    R[t,t^{-1}] \rTo 0
  \end{displaymath}
  is split exact, where $\iota_{q}$ and~${}_{p} \iota$ denote the
  inclusions, and where $\Delta$ is the ``diagonal'' map $r \mapsto
  (r,r)$; the splitting maps~$\rho$ and~$\sigma$ will be defined
  presently. --- The sequence can be re-written in more explicit
  terms:
  \begin{displaymath}
    0 \rTo \bigoplus_{k=-p}^{q} R_k \rTo^{\Delta} \bigoplus_{k \leq q}
    R_{k} \,\oplus\, \bigoplus_{k \geq -p} R_{k}
    \rTo[l>=4em]^{-\iota_q+ {}_p\iota} \bigoplus_{k \in \bZ} R_{k} \rTo 0
  \end{displaymath}
  The composition $(- \iota_q+ {_p\iota}) \circ \Delta$ is trivial. 
  We define~$\sigma$ by the formula
  \begin{displaymath}
    \sigma \colon \bigoplus_{k \in \bZ} R_{k} \rTo \bigoplus_{k \leq
      q} R_{k} \,\oplus\, \bigoplus_{k \geq -p} R_{k} \ , \quad
    \sum_{k \in \bZ} r_k\mapsto \bigg(-\sum_{k\leq q} r_{k}, \, \sum_{k
      \geq q+1} r_{k} \bigg)
  \end{displaymath}
  (note that $p+q \geq 0$ implies $q+1 > -p$) and $\rho$ by
  \begin{displaymath}
    \rho \colon \bigoplus_{k \leq q}
    R_{k} \,\oplus\, \bigoplus_{k \geq -p} R_{k}  \rTo
    \bigoplus_{k=-p}^{q} R_k \ , \quad \bigg(\sum_{k \leq q} r_k,\,
    \sum_{\ell \geq -p} s_\ell \bigg) \mapsto \sum_{\ell = -p}^{q}
    s_\ell \ .
  \end{displaymath}
  They satisfy the identities
  \begin{align*}
    \rho \circ \Delta & = \id \ , \\%
    \sigma \circ (- \iota_q+ {_p\iota}) + \Delta \circ \rho & = \id \
    , \\%
    (- \iota_q+ {_p\iota}) \circ \sigma & = \id \ ,
  \end{align*}
  as can be verified by direct calculation; thus the sequence is split
  exact as required.
 \end{proof}

\begin{corollary}
  \label{cor:H_of_O}
  If \(R[t,t^{-1}]\) is strongly graded then the cohomology chain
  complex \(H\big(\oo(q,p)\big)\) is \(R_0\)-finitely
  dominated.
\end{corollary}

\begin{proof}
  By Proposition~\ref{hogrs}, \(H\big(\oo(q,p)\big)\) is
  homotopy equivalent to a chain complex with one non-zero entry
  consisting of a finite sum of homogeneous components~\(R_k\)
  of~\(R[t,t^{-1}]\). Since the \(R_k\) are all finitely generated
  projective right \(R_0\)-modules by Proposition~\ref{fgpro},
  \(H\big(\oo(q,p)\big)\) is \(R_0\)-finitely dominated.
\end{proof}

\subsection*{Building chain complexes of pre-sheaves from chain
  complexes of modules}

Thanks to the graded structure of our ring~$R[t,t^{-1}]$ one can
extend a given chain complex of $R[t,t^{-1}]$-modules to a complex of
pre-sheaves. We start with the case of a single module homomorphism.

\begin{lemma}
  \label{lem:exmorts}
  Let $q,p \in \bZ$, and let \(f \colon R[t,t^{-1}]^n \rTo
  R[t,t^{-1}]^m\) be an \(R[t,t^{-1}]\)-linear map.
  For all sufficiently large numbers \(p^{\prime},
  q^{\prime} \in \bZ\) there exists a map of pre-sheaves
  \begin{displaymath}
    (f^-,f,f^+)\colon \bigoplus_{k=1}^{n} \oo(q,p) \rTo
    \bigoplus_{k=1}^{m}
    \oo(q^{\prime},p^{\prime}) \ ,
  \end{displaymath}
  depending on~$q^{\prime}$ and~$p^{\prime}$, which has
  the given~$f$ as its middle component. In other words, the module
  homomorphism~$f$ can be extended to a map of pre-sheaves.
\end{lemma}

\begin{proof}
  Consider the following diagram, where $q^{\prime}$
  and~$p^{\prime}$ are, for the moment, unspecified integers:
  \begin{diagram}
    \bigoplus_{k=1}^{n} t^{q}\cdot R[t^{-1}] &
    \rTo[l>=3em]^{\iota_{q}} & \bigoplus_{k=1}^{n} R[t,t^{-1}] &
    \lTo[l>=3em]^{_{p}\iota} & \bigoplus_{k=1}^{n} t^{-p}\cdot R[t]\\%
    && \dTo^{f} && \\%
    \bigoplus_{k=1}^{m} t^{q^{\prime}} \cdot R[t^{-1}] &
    \rTo[l>=3em]^{\iota_{q^{\prime}}} & \bigoplus_{k=1}^{m}
    R[t,t^{-1}] & \lTo[l>=3em]^{{}_{p^{\prime}}\iota} &
    \bigoplus_{k=1}^{m} t^{-p^{\prime}} \cdot R[t]
  \end{diagram}
  The map~$f$ yields $R[t,t^{-1}]$-linear maps ${}_k{f_j}\colon
  R[t,t^{-1}]\rTo R[t,t^{-1}]$ by restriction to the $k$th summand of
  the source and the $j$th summand of the target, and the (finite)
  collection of these maps determines~$f$. --- For now fix indices~$j$
  and~$k$. The element ${}_{k} f_{j}(1) \in R[t,t^{-1}]$ is a finite
  sum of non-zero homogeneous elements. Let $-a$ be the minimal
  occurring degree if ${}_{k} f_{j}(1) \neq 0$, and an arbitrary
  integer otherwise. As ${}_{k} f_{j}(r) = {}_{k} f_{j}(1) \cdot r$,
  the image of $t^{-p}\cdot R[t]$ under~${}_{k} f_{j}$ is contained in
  $t^{-(p+a)}\cdot R[t] \subseteq R[t,t^{-1}]$, hence is contained in
  $t^{-p^{\prime}}\cdot R[t]$ provided $p^{\prime}$ is
  sufficiently large in the sense that $p^{\prime} \geq
  a+p$. --- Allowing arbitrary indices $j$ and~$k$ now, we may choose
  $p^{\prime}$ sufficiently large for all~$j$ and~$k$. Then the
  map $f \circ {}_{p} \iota$ factors as
  \begin{displaymath}
    \bigoplus_{k=1}^{n} t^{-p}\cdot R[t] \rTo^{f^{+}}
    \bigoplus_{k=1}^{m} t^{-p^{\prime}} \cdot R[t]
    \rTo^{{_{p^{\prime}}} \iota} \bigoplus_{k=1}^{m}
    R[t,t^{-1}]
  \end{displaymath}
  where $f^{+}$ is actually the map~$f$, suitably restricted in source
  and target. --- The component $f^{-}$ is dealt with in a similar
  manner.
\end{proof}

\begin{proposition}[Extending chain complexes of modules to chain
  complexes of pre-sheaves]
  \label{exchcos}
  Let \(C\) be a bounded above chain complex of finitely generated
  free \(R[t,t^{-1}]\)-modules together with specified isomorphisms
  \(C_n\cong{R}[t,t^{-1}]^{k_n}\). Then $C$ is the ``middle''
  component of a chain complex of pre-sheaves. More precisely, there
  exists a chain complex of pre-sheaves \(\mathfrak{D} = \big( D^{-}
  \rTo D \lTo D^{+} \big)\) such that
  \begin{displaymath}
    \mathfrak{D}_n = \bigoplus_{k_{n}} \oo(q_n, p_n)
  \end{displaymath}
  for certain \(q_n,p_n\in\bZ\) with $q_{n} + p_{n} \geq 0$, with $D
  \cong C$ via the specified isomorphisms.

  In case $R[t,t^{-1}]$ is strongly graded, $\mathfrak{D}$ is a
  chain complex of sheaves in the sense of Definition~\ref{shfcond}.
\end{proposition}

\begin{proof}
  We identify the chain modules~$C_{n}$ with direct sums
  $R[t,t^{-1}]^{k_{n}}$ via the given isomorphisms. The boundary maps
  then take the form of homomorphisms $d_{n} \colon
  R[t,t^{-1}]^{k_{n}} \rTo R[t,t^{-1}]^{k_{n-1}}$.

  Let \(m\) be the maximal index of a non-zero entry of
  \(C\). Choose \(q_m = p_m = 0\). 

  Now for $\ell = m, m-1, m-2, \cdots$ we use
  Lemma~\ref{lem:exmorts} to extend the boundary map~$d_{\ell}$ to a
  map of pre-sheaves
  \begin{displaymath}
    \mathfrak{D}_{\ell} = \bigoplus_{k_{\ell}} \oo(q_{\ell},p_{\ell}) %
    \rTo[l>=5em]^{(d_{\ell}^{-}, d_{\ell}, d_{\ell}^{+})} \bigoplus_{k_{\ell-1}} %
    \oo(q_{\ell-1},p_{\ell-1}) = \mathfrak{D}_{\ell-1}
  \end{displaymath}
  with $q_{\ell-1} + p_{\ell-1} \geq 0$. 

  We have defined a (possibly infinite) sequence of maps of
  pre-sheaves $(d_{\ell}^{-}, d_{\ell}, d_{\ell}^{+})$. These maps are
  actually boundary maps of a chain complex of pre-sheaves. Indeed,
  $d_{\ell-1} \circ d_{\ell} = 0$ easily implies $d_{\ell-1}^{+} \circ
  d_{\ell}^{+} = 0$ and $d_{\ell-1}^{-} \circ d_{\ell}^{-} = 0$ as the
  structure maps of the diagrams $\oo(q_{\ell},p_{\ell})$ are
  injective.

  The last sentence of the Proposition holds as the
  pre-sheaves~$\oo(q,p)$ are actually sheaves, by
  Lemma~\ref{lem:sfcdn}, if $R[t,t^{-1}]$ is strongly graded.
\end{proof}

\subsection*{From trivial Novikov homology to finite domination}

With the machinery of sheaves set up we can implement the programme
of~\cite{TVB} to prove that trivial \textsc{Novikov} homology implies
finite domination. The strong grading proves to be crucial in two
places. It is the very fact that twisting sheaves are sheaves (and not
just pre-sheaves), combined with finiteness of their cohomology, that
makes the proof work.

\begin{notation}
  \label{notation:D}
  Given a chain complex \(\mathfrak{D}=(D^- \rTo D \lTo D^+)\) of
  pre-sheaves let \(\mathfrak{D^{+}}\) denote the diagram of chain
  complexes
  \begin{displaymath}
    \mathfrak{D}^{+} = \Big( D^+ \tensor_{R[t]} R[t,t^{-1}]  \rTo D^+ 
    \tensor_{R[t]} R\nov{t} \lTo D^+ \tensor_{R[t]} R\powers{t} \Big)
    \ ;
  \end{displaymath}
  similarly, let \(\mathfrak{D}^{-}\) denote the diagram
  \begin{displaymath}
    \mathfrak{D}^{-} = \Big(D^-\! \tensor_{R[t^{-1}]} R[t,t^{-1}]  \rTo
    D^-\! \tensor_{R[t^{-1}]} R\nov{t^{-1}} \lTo D^-\! \tensor_{R[t^{-1}]}
    R\powers{t^{-1}} \Big)\ . 
  \end{displaymath}
  In addition, we introduce the variants
  \begin{gather*}
        \mathfrak{D}^{\prime +} = \Big( D^+ \tensor_{R[t]} R[t,t^{-1}]
        \rTo 0 \lTo D^+ \tensor_{R[t]} R\powers{t} \Big) \\%
        \intertext{and}%
        \mathfrak{D}^{\prime -} = \Big( D^- \tensor_{R[t^{-1}]}
        R[t,t^{-1}] \rTo 0 \lTo D^- \tensor_{R[t^{-1}]}
        R\powers{t^{-1}} \Big) \ ,
  \end{gather*}
  and write $\zeta^{\pm} \colon \mathfrak{D}^{\pm} \rTo
  \mathfrak{D}^{\prime \pm}$ for the obvious maps of diagrams:
  \begin{diagram}[LaTeXeqno]
    \label{eq:zeta}
    D^{\pm} \tensor_{R[t^{\pm 1}]} R[t,t^{-1}] & \rTo[l>=2.5em] &
    D^{\pm} \tensor_{R[t^{\pm 1}]} R\nov{t^{\pm 1}} & \lTo[l>=2.5em] &
    D^{\pm} \tensor_{R[t^{\pm 1}]} R\powers{t^\pm 1} \\%
    \dTo<{\id} && \dTo<0 && \dTo<{\id} \\%
    D^{\pm} \tensor_{R[t^{\pm 1}]} R[t,t^{-1}] & \rTo & 0 &
    \lTo & D^{\pm} \tensor_{R[t^{\pm 1}]} R\powers{t^\pm 1}
  \end{diagram}
\end{notation}

\medbreak

We wish to analyse the hypercohomology complexes
of~$\mathfrak{D}^{\pm}$. To begin with, the sequence
\begin{equation}
\label{eq:R}
  0 \rTo R[t] \rTo^{%
    \Delta
  }%
  R[t,t^{-1}]\oplus R\powers{t} \rTo[l>=4em]^{%
    \rho
  }%
  R\nov{t}\rTo 0 \ ,
\end{equation}
where $\Delta(r) = (r,r)$ and $\rho(r,s) = s-r$, is split exact as a
sequence of right \(R_0\)-modules, with splitting maps
\begin{displaymath}
  R[t] \lTo^{\kappa} R[t,t^{-1}] \oplus R\powers{t} \lTo^{\lambda}
  R\nov{t} 
\end{displaymath}
specified by the formul\ae{}
\begin{align*}\kappa\colon\Big(\sum_{k\in\bZ}
  r_k,\sum_{k\geq 0} s_k\Big)\mapsto\sum_{k\geq 0} r_k \ ,\\
  \lambda\colon \sum_{k\geq n} r_k\mapsto \bigg(-\sum_{k<0} r_k,
  \sum_{k\geq 0} r_k\bigg) \ .
\end{align*}
Therefore the sequence~\eqref{eq:R} is exact (but not split) as a
sequence of \(R[t]\)-bimodules. If the complex \(D^+\)
consists of free $R[t]$-modules, tensoring~\eqref{eq:R} results in an
exact sequence of right $R[t]$-module chain complexes
\begin{multline*}
  0\rTo D^+ \tensor_{R[t]} R[t] \rTo D^+ \tensor_{R[t]} R[t,t^{-1}] \,
  \oplus \, D^+ \tensor_{R[t]} R\powers{t} \\ \rTo D^+ \tensor_{R[t]}
  R\nov{t} \rTo 0 \ .
\end{multline*}
This means that \(H^0(\mathfrak{D}^+) = D^+ \tensor_{R[t]} R[t] \cong
D^+\) and \(H^1(\mathfrak{D}^+) = 0\) (levelwise application
of~$H^{0}$ and~$H^{1}$). The latter implies that the natural map
\(\Delta^+\colon H^0(\mathfrak{D}^+)\rTo \hh{H}(\mathfrak{D}^+)\)
is a quasi-isomorphism \cite[Lemma~4.2]{TVB}. --- It can be shown by
analogous arguments that the natural map \(\Delta^- \colon
H^0(\mathfrak{D}^-) \rTo \hh{H}(\mathfrak{D}^-)\) is a
quasi-isomorphism, with source $D^{-} \tensor_{R[t^{-1}]} R[t^{-1}]
\cong D^{-}$, provided $D^{-}$ consists of free modules. We have
shown:

\begin{lemma}
  \label{lem:Delta}
  If $D^{+}$ consists of free $R[t]$-modules the map
  \begin{displaymath}
    \Delta^+ \colon H^0 (\mathfrak{D}^+) \rTo
    \hh{H}(\mathfrak{D}^+)
  \end{displaymath}
  is a quasi-isomorphism. Similarly, if $D^{-}$ consists of free
  $R[t^{-1}]$-modules the map $\Delta^- \colon
  H^0(\mathfrak{D}^-) \rTo \hh{H}(\mathfrak{D}^-)$ is a
  quasi-isomorphism.\qed
\end{lemma}

\medbreak

Now let us start with a bounded chain complex~$C$ of finitely
generated free $R[t,t^{-1}]$-modules. For each chain module $C_{n}
\neq \{0\}$ we choose an isomorphism with $R[t,t^{-1}]^{k_{n}}$. Let
$\mathfrak{D} = \big( D^{-} \rTo D \lTo D^{+}\big)$ denote the
resulting complex of pre-sheaves according to
Proposition~\ref{exchcos}, and let $\mathfrak{D}^{\pm}$
and~$\mathfrak{D}^{\prime \pm}$ be the diagrams defined at the
beginning of this section.--- The structure map $D \lTo D^{+}$ has an
$R[t,t^{-1}]$-linear adjoint, $D \lTo D^{+} \tensor_{R[t]}
R[t,t^{-1}]$, which induces a map of diagrams
\begin{displaymath}
  \mathfrak{D}^{+} \quad \rTo \quad \Big( 0 \rTo 0 \lTo D \Big) \ ;
\end{displaymath}
upon application of~$\hh{H}$ this yields a map
$\pi^{+} \colon \hh{H}(\mathfrak{D}^{+}) \rTo D$. We have
similarly a map $\pi^{-} \colon \hh{H}(\mathfrak{D}^{-}) \rTo D$,
and analogous maps using $\mathfrak{D}^{\prime \pm}$
denoted~$\pi^{\prime \pm}$. All these fit into the commutative diagram
displayed in Fig.~\ref{fig:big}.

\begin{figure}[ht]
  \centering
  \begin{diagram}[LaTeXeqno]
    \label{eq:big}
    D^- &\rTo &D &\lTo&D^+\\
    \dTo^{\Delta^-} &&\dTo^{\id} &&\dTo^{\Delta^+}\\
    \hh{H}(\mathfrak{D}^-)&\rTo^{\pi^-}&D&\lTo^{\pi^+}&\hh{H}(\mathfrak{D}^+)\\
    \dTo^{\hh{H}(\zeta^-)} &&\dTo^{\id} &&\dTo^{\hh{H}(\zeta^+)}\\
    \hh{H}(\mathfrak{D}^{\prime-})&\rTo^{\pi^{\prime-}} &D
    &\lTo^{\pi^{\prime+}}&\hh{H}(\mathfrak{D}^{\prime+})
  \end{diagram}
  \caption{Commutative diagram}
  \label{fig:big}
\end{figure}

\begin{lemma}
  \label{lem:Hzeta}
  If $R[t,t^{-1}]$ is strongly graded, and if the two complexes $C
  \tensor_{R[t,t^{-1}]} R\nov{t}$ and $C \tensor_{R[t,t^{-1}]}
  R\nov{t^{-1}}$ have trivial homology, then the maps
  $\hh{H}(\zeta^-)$ and~$\hh{H}(\zeta^+)$ are
  quasi-isomorphisms.
\end{lemma}

\begin{proof}
  There is a chain of isomorphisms
  \begin{multline*}
    D^{\pm} \tensor_{R[t^{\pm 1}]} R\nov{t^{\pm 1}} \cong D^{\pm}
    \tensor_{R[t^{\pm 1}]} R[t,t^{-1}] \tensor_{R[t,t^{-1}]}
    R\nov{t^{\pm 1}} \\ \cong D \tensor_{R[t,t^{-1}]} R\nov{t^{\pm 1}}
    \cong C \tensor_{R[t,t^{-1}]} R\nov{t^{\pm 1}} \ ,
  \end{multline*}
  the second one due to the fact that $D$ is a sheaf in the strongly
  graded setting. By hypothesis the last complex is acyclic. This
  means that all vertical maps in the diagram~\eqref{eq:zeta} are
  quasi-isomorphisms, that is, $\zeta^{\pm}$ consists of
  quasi-isomorphisms. Hence application of~$\hh{H}$ results in a
  quasi-isomorphism~$\hh{H}(\zeta^{\pm})$
  by~\cite[Lemma~4.2]{TVB}.
\end{proof}

Recall that, by construction, $\mathfrak{D}_{n}$ is a finite direct
sum of diagrams of the form~$\oo(q,p)$, with $q+p \geq 0$. It follows
from the calculation in Proposition~\ref{hogrs} that $H^{1}
(\mathfrak{D}) = 0$ (levelwise) so that the inclusion $H^{0}
(\mathfrak{D}) \rTo \hh{H}(\mathfrak{D})$ is a quasi-isomorphism
\cite[Lemma~4.2]{TVB}. With Proposition~\ref{fgpro} this yields the
following result:

\begin{lemma}
  \label{lem:H0}
  The bounded chain complex $H^{0} (\mathfrak{D})$ is quasi-isomorphic
  to the complex $\hh{H}(\mathfrak{D})$. If $R[t,t^{-1}]$ is
  strongly graded, $H^{0} (\mathfrak{D})$ consists of finitely
  generated projective $R_{0}$-modules.\qed
\end{lemma}

\begin{proof}[Proof of Theorem~\ref{thm:main}, ``if'' part]
  As before we start with a bounded chain complex~$C$ of finitely
  generated free $R[t,t^{-1}]$-modules, and construct a complex of
  sheaves $\mathfrak{D} = \big( D^{-} \rTo D \lTo D^{+}\big)$
  according to Proposition~\ref{exchcos}, with $D \cong C$. We will
  also use the diagrams $\mathfrak{D}^{\pm}$ and~$\mathfrak{D}^{\prime
    \pm}$ as defined at the beginning of this section.

  Our hypothesis now is that $R[t,t^{-1}]$ is strongly graded. In this
  situation all the vertical maps in diagram~\eqref{eq:big} are
  quasi-isomorphisms, by Lemmas~\ref{lem:Delta}
  and~\ref{lem:Hzeta}. So by applying~$\hh{H}$ to the rows of
  the diagram we obtain a chain of maps
  \begin{equation}
    \label{eq:qi}
    H^{0}(\mathfrak{D}) \rTo^{\simeq} \hh{H}(\mathfrak{D})
    \rTo^{\simeq} \hh{H} \Big( \hh{H}(\mathfrak{D}^{\prime-})
    \rTo^{\pi^{\prime-}}  D \lTo^{\pi^{\prime+}} \hh{H}
    (\mathfrak{D}^{\prime+}) \Big) \ ;
  \end{equation}
  the first one is a quasi-isomorphism by Lemma~\ref{lem:H0}, the
  second because the functor $\hh{H}$ preserves quasi-isomorphisms
  \cite[Lemma~4.2]{TVB}.

  By explicitly spelling out the definitions, we see that the chain
  complex $\hh{H} \big( \hh{H} (\mathfrak{D}^{\prime-})
  \rTo^{\pi^{\prime-}} D \lTo^{\pi^{\prime+}} \hh{H}
  (\mathfrak{D}^{\prime+}) \big)$ contains the complex
  \begin{displaymath}
    \hh{H} \Big( D^{-} \tensor_{R[t^{-1}]} R[t,t^{-1}] \rTo D
    \lTo D^{+} \tensor_{R[t]} R[t,t^{-1}] \Big)
  \end{displaymath}
  as a retract. But the diagram~$\mathfrak{D}$ is a sheaf, making use
  of the strong grading again, so the maps $D^{\pm} \tensor_{R[t^{\pm
      1}]} R[t,t^{-1}] \rTo D$ are isomorphisms. It follows that the
  previous chain complex is isomorphic to $\hh{H} \big( D \rTo^{=}
  D \lTo^{=} D \big)$, and thus quasi-isomorphic to~$D \cong C$.

  Combined with~\eqref{eq:qi}, we thus see that in the derived
  category of~$R_{0}$ the complex~$C$ is a retract
  of~$H^{0}(\mathfrak{D})$. Both are bounded complexes of
  $R_{0}$-projective modules, the former by
  Corollary~\ref{cor:stays_projective}, the latter by
  Lemma~\ref{lem:H0}. It follows from general theory of derived
  categories that there are chain maps $\alpha \colon C \rTo H^{0}
  (\mathfrak{D})$ and $\beta \colon H^{0}(\mathfrak{D}) \rTo C$ with
  $\beta\alpha \simeq \id$. As $H^{0 }(\mathfrak{D})$ consists of
  finitely generated projective $R_{0}$-modules (Lemma~\ref{lem:H0}
  again), this proves that $C$~is $R_{0}$\nobreakdash-finitely
  dominated as desired.
\end{proof}

\section{Finite domination implies trivial Novikov homology}

From now on, and for the remainder of the paper, we suppose that the
$\bZ$-graded ring~$R[t,t^{-1}]$ admits a partition of unity
\(1=\sum_{j} x^{(-1)}_jy^{(1)}_j\) of type~\((-1,1)\), which we choose
once and for all.

\subsection*{Canonical resolution and algebraic tori}

For a given \(R[t,t^{-1}]\)-module~$C$, or a given chain complex~$C$
of such modules, we use the chosen partition of unity to define an
$R[t,t^{-1}]$-linear map
\begin{equation}
  \label{eq:mu}
  \mu \colon C\tensor_{R_0} R[t,t^{-1}] \rTo C\tensor_{R_0}
  R[t,t^{-1}] \ , \quad\!  c \tensor r \mapsto c \tensor r -
  \sum_{j} c x^{(-1)}_j \tensor y^{(1)}_j r \ .
\end{equation}
Note that for any $s \in R_{0}$ and any partition of unity $1 =
\sum_{\ell} u_{\ell} v_{\ell}$ of type~$(-1,1)$ there are equalities
\begin{multline*}
  \sum_{j} c x^{(-1)}_j \tensor y^{(1)}_j sr = \sum_{j, \ell} c
  x^{(-1)}_j \tensor y^{(1)}_j s u_\ell v_\ell r \\ %
  = \sum_{\ell, j} c x^{(-1)}_j y^{(1)}_j s u_\ell \tensor
  v_\ell r = \sum_{\ell} c s u_\ell \tensor v_\ell
  r \ .
\end{multline*}
Specialising to $u_{\ell} = x^{(-1)}_{\ell}$ and $v_{\ell} =
y^{(1)}_{\ell}$ yields that the map~$\mu$ is
$R_{0}$\nobreakdash-balanced, and hence well-defined. On the other
hand, specialising to $s=1$ shows that, contrary to appearance, the
map~$\mu$ does actually not depend on the choice of partition of
unity. --- It might be worth pointing out that the map~$\mu$ cannot
be defined in the absence of additional data; the strongly graded
structure of the ring enters the picture in a rather subtle form here.

\begin{proposition}[Canonical resolution]
  \label{splites}
  For any $R[t,t^{-1}]$-module~$M$ there is a sequence of
  \(R[t,t^{-1}]\)-modules
  \begin{equation}
    \label{eq:can_res}
    0\rTo M\tensor_{R_0} R[t,t^{-1}] \rTo^{\mu} M\tensor_{R_0}
    R[t,t^{-1}] \rTo^{\pi} M\rTo 0 \ ,
  \end{equation}
  where $\pi (m \tensor r) = mr$ and $\mu$~is as in~\eqref{eq:mu}. The
  sequence is natural in~$M$. If $R[t,t^{-1}]$ is strongly graded then
  the sequence is split exact as a sequence of right \(R_0\)-modules,
  and hence is exact (but possibly non-split) as a sequence of right
  \(R[t,t^{-1}]\)-modules.
\end{proposition}

\begin{proof}
  We first note that \(\pi\mu=0\) as
  \begin{multline*}
    \pi\mu(m \tensor r) = \pi \big( m\tensor r -\sum_{j} m
    x^{(-1)}_j\tensor y^{(1)}_j r \big) \\ = m r -\sum_{j} m
    x^{(-1)}_j y^{(1)}_j r = mr - m1r = 0 \ .
  \end{multline*}

  Let us now suppose that $R[t,t^{-1}]$ is strongly graded. In
  addition to our fixed partition of unity $1 = \sum_{\ell_{1}}
  x^{(-1)}_{\ell_{1}} y^{(1)}_{\ell_{1}}$ we choose for all
  \(k\in\bZ\), $k \neq 1$, a partition of unity $1 = \sum_{\ell_{k}}
  x^{(-k)}_{\ell_{k}} y^{(k)}_{\ell_{k}}$ of type \((-k,k)\); as
  before this is understood to be a finite sum with
  $x^{(-k)}_{\ell_{k}} \in R_{-k}$ and $y^{(k)}_{\ell_{k}} \in
  R_{k}$. Such partitions of unity exist by
  Proposition~\ref{prop:characterisation_strongly_graded}.

  We denote by~$\iota$ the right $R_{0}$-linear map
  \begin{displaymath}
    \iota \colon M \rTo M\tensor_{R_0} R[t,t^{-1}] \ ,\quad  m
    \mapsto m \tensor 1 \ ;
  \end{displaymath}
  clearly $\pi \iota = \id_{M}$. Next, we define an $R_{0}$-linear map
  \begin{displaymath}
    \tau \colon M\tensor_{R_0} R[t,t^{-1}] \rTo M \tensor_{R_0}
    R[t,t^{-1}] \ ;
  \end{displaymath}
  as $M\tensor_{R_0} R[t,t^{-1}] \cong \bigoplus_{n \in \bZ} M
  \tensor_{R_{0}} R_{n}$ as a right $R_{0}$-module it will be enough
  to specify the restrictions $\tau_{n} = \tau|_{M \tensor
    R_{n}}$. For $m \in M$ and $r_{n} \in R_{n}$ these are given by
  \begin{displaymath}
    \tau_{n} (m \tensor r_{n}) =
    \begin{cases}
      - \sum_{k=1}^{n\hphantom{--1}} \sum_{\ell_k} \big( m
      x^{(k)}_{\ell_k} \tensor y^{(-k)}_{\ell_k} r_n \big) & \text{if
      } n > 0 \ , \\%
      \hphantom{-} 0 & \text{if } n = 0 \ , \\%
      \hphantom{-} \sum_{k=0}^{-n-1} \sum_{j_{k}} \big( m
      x^{(-k)}_{j_{k}} \tensor y^{(k)}_{j_{k}} r_{n} \big) & \text{if
      } n < 0 \ .
    \end{cases}
  \end{displaymath}

  The map $\tau$ satisfies $\tau \mu = \id$; we will verify $\tau \mu
  (m \tensor r_{n}) = m \tensor r_{n}$ for $n \geq 1$, the case $n
  \leq 0$ being similar. So let $m \in M$ and $r_{n} \in R_{n}$, for
  some $n \geq 1$. Then
  \begin{multline*}
    \tau \mu (m \tensor r_{n}) = \tau \Big( m\tensor r_n - \sum_{j} m
    x^{(-1)}_{j}\tensor y^{(1)}_{j} r_n\Big) \\ %
    = \tau(m\tensor r_n) -\tau\Big(\sum_{j} m x^{(-1)}_{j}\tensor
    y^{(1)}_{j} r_n\Big) \\ %
    = \tau_n(m\tensor r_n) -\tau_{n+1}\Big(\sum_{j} m
    x^{(-1)}_{j}\tensor y^{(1)}_{j} r_n\Big) \ .
  \end{multline*}
  Now by definition
  \begin{displaymath}
    \tau_n(m\tensor r_n)=-\sum_{k=1}^{n} \sum_{\ell_k} \big( m
    x^{(k)}_{\ell_k} \tensor y^{(-k)}_{\ell_k} r_n \big)
  \end{displaymath}
  while
  \begin{displaymath}
    \tau_{n+1} \Big( \sum_{j} m x^{(-1)}_{j}\tensor y^{(1)}_{j} r_n \Big) 
    = - \sum_{k=1}^{n+1} \Big( \sum_{\ell_k} \sum_{j} m x^{(-1)}_{j}
    x^{(k)}_{\ell_k} \tensor y^{(-k)}_{\ell_k} y^{(1)}_{j} r_n \Big) \
    \! .
  \end{displaymath}
  The last term in parentheses, for any fixed~\(k\), can be
  simplified:
  \begin{align*}
    \sum_{\ell_k} \sum_{j} & m x^{(-1)}_{j}x^{(k)}_{\ell_k}\tensor
    y^{(-k)}_{\ell_k} y^{(1)}_{j} r_n \\ %
    & \underset{(\ddagger)}= \sum_{\ell_k} \sum_{j} m x^{(-1)}_{j}
    x^{(k)}_{\ell_k} \tensor y^{(-k)}_{\ell_k} y^{(1)}_{j} \cdot %
    \Big( \sum_{\ell_{k-1}} x^{(k-1)}_{\ell_{k-1}}
    y^{(-k+1)}_{\ell_{k-1}} \Big) \cdot r_{n} \\ %
    & = \sum_{\ell_{k-1}} \sum_{\ell_k} \sum_{j} m x^{(-1)}_{j}
    x^{(k)}_{\ell_k} \tensor y^{(-k)}_{\ell_k} y^{(1)}_{j}
    x^{(k-1)}_{\ell_{k-1}} y^{(-k+1)}_{\ell_{k-1}} r_n \\%
    & \underset{(\dagger)}= \sum_{\ell_{k-1}} \sum_{\ell_k} \sum_{j} m
    x^{(-1)}_{j} x^{(k)}_{\ell_k} y^{(-k)}_{\ell_k} y^{(1)}_{j}
    x^{(k-1)}_{\ell_{k-1}} \tensor y^{(-k+1)}_{\ell_{k-1}} r_n \\ %
    & \underset{(\ddagger)}= \sum_{\ell_{k-1}} m
    x^{(k-1)}_{\ell_{k-1}} \tensor  y^{(-k+1)}_{\ell_{k-1}} r_n
  \end{align*}
  where at~$(\dagger)$ we have used that $y^{(-k)}_{\ell_k}
  y^{(1)}_{j} x^{(k-1)}_{\ell_{k-1}} \in R_{0}$, and at~$(\ddagger)$
  we have used that $\sum_{\ell_{k}} x^{(k)}_{\ell_{k}}
  y^{(-k)}_{\ell_{k}} = \sum_{\ell_{k-1}} x^{(k-1)}_{\ell_{k-1}}
  y^{(-k+1)}_{\ell_{k-1}} = \sum_{j} x^{(-1)}_{j} y^{(1)}_{j} = 1$. It
  follows together with the previous expressions that $\tau \mu (m
  \tensor r_{n})$ equals
  \begin{multline*}
    % \tau \mu (m \tensor r_{n}) \\ %
    -\sum_{k=1}^{n} \sum_{\ell_k} \big( m x^{(k)}_{\ell_k} \tensor
    y^{(-k)}_{\ell_k} r_n \big) + \sum_{k=1}^{n+1} \sum_{\ell_{k-1}}
    \big( m x^{(k-1)}_{\ell_{k-1}} \tensor y^{(-k+1)}_{\ell_{k-1}} r_n
    \big) \\ %
    = \sum_{\ell_{0}}  m x^{(0)}_{\ell_{0}} \tensor y^{(0)}_{\ell_{0}}
    r_{n} = m \tensor r_{n} \ .
  \end{multline*}

  \medbreak

  To show that our sequence~\eqref{eq:can_res} is split exact when
  considered as a sequence of~$R_{0}$\nobreakdash-modules it remains
  only to prove that \(\mu\tau+\iota\pi=\id_{M\tensor
    R[t,t^{-1}]}\). The calculation is similar to the one just
  finished, making use of existence of partitions of unity in exactly
  the same manner. We omit the details.
\end{proof}

\begin{corollary}
  \label{cor:is_torus}
  For any chain complex~\(C\) of \(R[t,t^{-1}]\)-modules there is a
  quasi-isomorphism \(\cone(\mu) \rTo^{\sim} C\).
\end{corollary}

\begin{proof}
  By the previous Proposition there is a short exact sequence of chain
  complexes
  \begin{displaymath}
    0 \rTo C \tensor_{R_{0}} R[t,t^{-1}] \rTo^{\mu} C \tensor_{R_{0}}
    R[t,t^{-1}] \rTo^{\pi} C \rTo 0 \ .
  \end{displaymath}
  Thus the canonical map $\cone(\mu) \rTo C$ is a quasi-isomorphism.
\end{proof}

\begin{definition}
  The mapping cone of~$\mu$ in the previous Corollary is called the
  \emph{algebraic torus of~\(C\)} and denoted~\(\T(C)\).
\end{definition}

\subsection*{The Mather trick for the algebraic torus}

Let \(C\) be an \(R[t,t^{-1}]\)-module chain complex, and let \(D\) be
an \(R_0\)-module chain complex. Let $\alpha \colon C \rTo D$ and
$\beta \colon D \rTo C$ be $R_{0}$-linear chain maps and $H$ a chain
homotopy such that $H \colon \beta \alpha\simeq \id_{C}$; that is, $dH
+ Hd = \id_{C} - \beta \alpha$ where $d$~is the differential of~$C$.
Define
\begin{equation}
  \label{eq:map_nu}
  \nu\colon D \tensor_{R_0} R[t,t^{-1}]\rTo D\tensor_{R_0}
  R[t,t^{-1}]
\end{equation}
by the formula
$\nu=(\alpha\tensor\id)\circ\mu\circ(\beta\tensor\id)$. Then the
diagram
\begin{diagram}
  C \tensor_{R_0} R[t,t^{-1}]&\rTo[l>=3em]^{\mu} &C\tensor_{R_0}
  R[t,t^{-1}] \\%
  \dTo^{\alpha\tensor \id}&&\dTo^{\alpha\tensor \id}\\%
  D \tensor_{R_0} R[t,t^{-1}]& \rTo^{\nu} &D\tensor_{R_0} R[t,t^{-1}]
\end{diagram}
is homotopy commutative with homotopy
\begin{displaymath}
  J = (\alpha \tensor \id) \circ \mu \circ (H \tensor \id) \colon \nu
  \circ (\alpha \tensor \id) \simeq (\alpha \tensor \id) \circ \mu \ .
\end{displaymath}
This homotopy induces a preferred map of $R[t,t^{-1}]$-module chain
complexes
\begin{displaymath}
  \alpha_{*} = \begin{pmatrix} \alpha \tensor \id & 0 \\ J & \alpha
    \tensor \id \end{pmatrix} \colon \T(C) = \cone (\mu) \rTo
  \cone(\nu) \ .
\end{displaymath}
If $\alpha$ is a quasi-isomorphism and $R[t,t^{-1}]$ is strongly
graded then $\alpha \tensor \id$ is a quasi-isomorphism as well;
indeed, the functor $\,\cdot\, \tensor_{R_{0}} R[t,t^{-1}]$ is exact
in the strongly graded case by Proposition~\ref{fgpro}. We obtain the
following result analogous to the \textsc{Mather} trick in the
topological context \cite[``Whitehead Lemma'', \S2]{RFDNR}:

\begin{lemma}[\textsc{Mather} trick]
  \label{lem:compare_mu_nu}
  Let \(C\) be an \(R[t,t^{-1}]\)-module chain complex, and let \(D\)
  an \(R_0\)-module chain complex. Let $\alpha \colon C \rTo D$ and
  $\beta \colon D \rTo C$ be $R_{0}$-linear chain maps such that
  $\beta \alpha \simeq \id_{C}$ via a specified homotopy. Then there
  is a preferred map $\alpha_{*} \colon \T(C) \rTo \cone(\nu)$.  If in
  addition $\alpha$ is a quasi-isomorphism and $R[t,t^{-1}]$ is
  strongly graded, $\alpha_{*} \colon \T(C) \rTo \cone(\nu)$ is a
  quasi-isomorphism.\qed
\end{lemma}

\begin{corollary}\label{findomimp}
  Suppose $R[t,t^{-1}]$ is strongly graded. Given a bounded below
  chain complex~\(C\) of projective \(R[t,t^{-1}]\)-modules, a bounded
  below chain complex~\(D\) of projective \(R_0\)\nobreakdash-modules,
  and an $R_{0}$-homotopy equivalence \(\alpha \colon C \rTo^{\simeq}
  D\), there is a homotopy equivalence
  \begin{displaymath}
    C \tensor_{R[t,t^{-1}]} R\nov{t^{\pm 1}} \simeq \cone(\nu)
    \tensor_{R[t,t^{-1}]} R\nov{t^{\pm 1}} \ .
  \end{displaymath}
\end{corollary}

\begin{proof}
  From the previous Lemma and Corollary~\ref{cor:is_torus} we know
  that there are quasi-isomorphisms $C \lTo \T(C) \rTo^{\alpha_{*}}
  \cone(\nu)$. As both $C$ and~$\cone(\nu)$ are bounded below and
  consist of projective $R[t,t^{-1}]$-modules, these two complexes are
  actually homotopy equivalent. As taking tensor products preserves
  homotopy equivalences we have proven the claim.
\end{proof}

\subsection*{Bicomplexes and truncated powers}

We extend our portfolio of homological techniques further by
re-writing the complex~$\cone(\nu)$ as the totalisation of a
bicomplex, and by introducing twisted truncated powers.

Let \(C\) be an \(R[t,t^{-1}]\)-module chain complex, and let \(D\) an
\(R_0\)-module chain complex. Let $\alpha \colon C \rTo D$ and $\beta
\colon D \rTo C$ be $R_{0}$-linear chain maps. Define
$\nu=(\alpha\tensor\id)\circ\mu\circ(\beta\tensor\id)$ as
in~\eqref{eq:map_nu}. Let $\zeta_{n,m}$ denote the $R_{0}$-linear map
\begin{displaymath}
  D_{m} \tensor_{R_{0}} R_{n} \rTo D_{m} \tensor_{R_{0}} R_{n+1} \ ,
  \quad z \tensor r \mapsto \sum_{j} \alpha \big( \beta(z)
  x^{(-1)}_{j} \big) \tensor y^{(1)}_{j} r \ ,
\end{displaymath}
and let $E_{\bullet, \bullet}$ denote the bicomplex of right
$R_{0}$-modules given by
\begin{align}
  E_{n,m} = \big( D_{n+m-1} \tensor_{R_0} R_{-n} \big) %
  \, & \oplus\, %
  \big( D_{n+m} \tensor_{R_0} R_{-n} \big) %
  \label{eq:E} \\ %
\intertext{with differentials}%
  \notag d_{H} =
  \begin{pmatrix}
    0                       & 0 \\
    \zeta_{-n, n+m}  & 0
  \end{pmatrix}
  & \colon E_{n,m} \rTo E_{n-1, m} \\%
  \intertext{and}%
  d_{V} =
  \begin{pmatrix}
    -d \tensor \id           & 0             \\
    \alpha \beta \tensor \id & d \tensor \id
  \end{pmatrix}
  & \colon E_{n, m} \rTo E_{n, m-1}
  \label{eq:dV}
\end{align}
where $d$ is the differential of the chain complex~$D$.

The {\it totalisation\/} $\tot(E_{\bullet, \bullet})$ is the chain
complex with $\tot(E_{\bullet, \bullet})_{\ell} = \bigoplus_{n+m =
  \ell} E_{n,m}$ and differential $d_{H} + d_{V}$. More explicitly, we
have an identification
\begin{multline*}
  \tot(E_{\bullet, \bullet})_{\ell} = \bigoplus_{n \in \bZ}
  E_{-n,\ell+n} = \bigoplus_{n \in \bZ} \Big( \big( D_{\ell-1}
  \tensor_{R_0} R_{n} \big) \Big)  \,\oplus\,  \big( D_{\ell}
  \tensor_{R_0} R_{n} \big) \\ %
  = \big( D_{\ell-1} \tensor_{R_{0}} R[t,t^{-1}] \big) \,\oplus\, \big(
  D_{\ell} \tensor_{R_{0}} R[t,t^{-1}] \big) \ ,
\end{multline*}
under which the differential $d = d_{H} + d_{V}$ coincides with the
differential of~$\mathrm{cone}(\nu)$. A straightforward calculation
then shows that $d_{H}$ and~$d_{V}$ are anti-commuting
differentials. We summarise the construction:

\begin{lemma}
  The data listed above yields a bicomplex in the sense that $d_{H}
  \circ d_{H} = 0$, $d_{H} \circ d_{V} = - d_{V} \circ d_{H}$ and
  $d_{V} \circ d_{V} = 0$.  Its totalisation~$\tot(E_{\bullet,
    \bullet})$ is isomorphic to~$\cone(\nu)$.\qed
\end{lemma}

We wish to analyse the tensor product $\cone(\nu)
\tensor_{R[t,t^{-1}]} R\nov{t}$ using the bicomplex above. For this,
we need to digress a little and talk about truncated powers, or rather
a ``twisted'' version thereof that takes the graded structure of the
ring into account.

\begin{definition}
  Given a right \(R_0\)-module \(M\), we define the {\it twisted left
    truncated power of~$M$}, denoted $\tprodlt M$, by
  \begin{displaymath}
    \tprodlt M =
    \bigoplus_{n<0} \big(M \tensor_{R_0} R_n\big) \,\oplus\,
    \prod_{{n\geq 0}} \big( M \tensor_{R_{0}} R_n\big) \ ,
  \end{displaymath}
  and the {\it twisted right truncated power of~$M$}, denoted
  $\tprodrt M$, by
  \begin{displaymath}
    \tprodrt M = \prod_{n \leq 0} \big( M \tensor_{R_0} R_n \big)
    \,\oplus\, \bigoplus_{n > 0} \big( M\tensor_{R_0} R_n \big) \ .
  \end{displaymath}
\end{definition}

We note that $\tprodlt M$ has a right $R\nov{t}$-module structure; if
we write elements of~$\tprodlt M$ as formal \textsc{Laurent} series
$\sum_{n \geq m} xt^{n}$ with $x_{n} \in M \tensor_{R_{0}} R_{n}$ and
elements of~$R\nov{t}$ as formal \textsc{Laurent} series $\sum_{n \geq
  p} r_{n} t^{n}$ with $r_{n} \in R_{n}$, it is given by the obvious
multiplication of series formula using $x_{k} r_{n} \in M
\tensor_{R_{0}} R_{k+n}$ via the assignment $(m \tensor s_{k}) \cdot
r_{n} = m \tensor (s_{k} r_{n})$. --- Similarly, $ \tprodrt M$
carries a natural right $R\nov{t^{-1}}$-module structure.

\begin{proposition}
  \label{finpres}
  For a finitely presented \(R_0\)-module \(M\), there is an
  isomorphism of right $R\nov{t}$-modules
  \begin{align*}
    \Phi_M \colon M \tensor_{R_{0}} R\nov{t} & \rTo \tprodlt M \ ,
    \quad m \tensor \sum_{k} r_{k} t^{k} \mapsto \sum_{k} (m \tensor
    r_{k}) t^{k}\\%
    \intertext{and an isomorphism of right \(R\nov{t^{-1}}\)-modules}%
    \Psi_M \colon M \tensor_{R_0} R\nov{t^{-1}} & \rTo \tprodrt M \ ,
    \quad m \tensor \sum_{k} r_{k} t^{k} \mapsto \sum_{k} (m \tensor
    r_{k}) t^{k} \ .
  \end{align*}
  Both isomorphisms are natural in~$M$.
\end{proposition}

\begin{proof}
  We show that $\Phi_{M}$ is bijective, the case of~$\Psi_{M}$ being
  similar. --- Suppose first that $M = F$ is free on the basis $e_{1},
  e_{2}, \cdots, e_{q}$. Then $F \tensor_{R_{0}} R\nov{t}$ is a free
  $R\nov{t}$-module with basis $e_{1} \tensor 1, e_{2} \tensor 1,
  \cdots, e_{q} \tensor 1$. Thus any $x \in F \tensor_{R_{0}}
  R\nov{t}$ can uniquely be written in the form
  \begin{displaymath}
    x = \sum_{j=1}^{q} \Big( e_{j} \tensor \sum_{k} r_{jk} t^{k} \Big)
  \end{displaymath}
  with $r_{jk} \in R_{k}$, and $r_{jk} = 0$ if $k$~is sufficiently
  small. Suppose that $x \in \ker \Phi_{F}$ so that
  \begin{displaymath}
    0 = \Phi_{F} (x) = \sum_{j=1}^{q} \sum_{k} (e_{j} \tensor r_{jk})
    t^{k} = \sum_{k} \sum_{j=1}^{q} (e_{j} \tensor r_{jk}) t^{k}
  \end{displaymath}
  in the twisted left truncated power of~$F$. This implies the
  equality $\sum_{j=1}^{q} e_{j} \tensor r_{jk} = 0 \in F
  \tensor_{R_{0}} R_{k} \subseteq F \tensor_{R_{0}} R[t,t^{-1}]$ for
  all~$k$; as the last module is free on basis elements $e_{j} \tensor
  1$ we conclude that $r_{jk} = 0$ for all~$k$ and~$j$. Consequently
  $x=0$ which proves that $\Phi_{F}$ is injective.

  Now let $z = \sum_{k \geq n} z_{k} t^{k} \in \tprodlt F$ with $z_{k}
  \in F \tensor_{R_{0}} R_{k}$; using that $F$~is free on basis
  elements~$e_{j}$ as before we see that we can write $z_{k}$ in the
  form $z_{k} = \sum_{j} e_{j} \tensor z_{jk}$ with $z_{jk} \in
  R_{k}$. Then
  \begin{displaymath}
    x = \sum_{j} \Big( e_{j} \tensor \sum_{k} z_{jk} t^{k} \Big)
  \end{displaymath}
  is an element of~$F \tensor_{R_{0}} R\nov{t}$ satisfying
  $\Phi_{F}(x) = z$. Thus $\Phi_{F}$ is seen to be surjective.

  For the general case consider a presentation \(G\rTo F\rTo M\rTo 0\)
  of~$M$ by finitely generated free modules~$F$ and~$G$; standard
  homological algebra, using that the functors $X \mapsto X
  \tensor_{R_{0}} R\nov{t}$ and $X \mapsto \tprodlt X$ are right
  exact, shows that $\Phi_{M}$ is bijective,
  cf.~\cite[Lemma~2.1]{DCVNC}.
\end{proof}

The {\it right truncated totalisation} of~$E_{\bullet, \bullet}$,
denoted $\totrt(E_{\bullet, \bullet})$, is the chain complex with
\begin{displaymath}
  \totrt(E_{\bullet, \bullet})_{\ell} = \prod_{n \leq 0} E_{n, \ell-n}
  \oplus \bigoplus_{n \geq 0} E_{n, \ell-n}
\end{displaymath}
and differential $d_{H} + d_{V}$. Plugging in the definition
of~$E_{n,m}$ this can be re-written as
\begin{align*}
  \totrt(E_{\bullet, \bullet})_{\ell} = & \hphantom{\oplus} \ \,
  \prod_{n \leq 0} \Big( \big( D_{\ell-1} \tensor_{R_0} R_{-n} \big)
  \,\oplus\, \big( D_{\ell} \tensor_{R_0} R_{-n} \big) \Big) \\ %
  & \oplus\, \bigoplus_{n > 0} \Big( \big( D_{\ell-1} \tensor_{R_0}
  R_{-n} \big) \,\oplus\, \big( D_{\ell} \tensor_{R_0} R_{-n} \big)
  \Big) \\ %
  = & \tprodlt D_{\ell-1} \,\oplus\, \tprodlt D_{\ell} \ ;
\end{align*}
if the complex~$D$ consists of finitely presented $R_{0}$-modules we
can thus use Proposition~\ref{finpres} to identify $\totrt(E_{\bullet,
  \bullet})_{\ell}$ with the module $\big( D_{\ell-1} \tensor_{R_{0}}
R\nov{t} \big) \oplus \big( D_{\ell} \tensor_{R_{0}} R\nov{t}
\big)$. When combined with the isomorphisms $\cone(\nu)
\tensor_{R[t,t^{-1}]} R\nov{t} \cong \cone(\nu \tensor
\id_{R\nov{t}})$ and
\begin{displaymath}
  D_{\ell} \tensor_{R_{0}} R[t,t^{-1}] \tensor_{R[t,t^{-1}]} R\nov{t}
  \cong D_{\ell} \tensor_{R_{0}} R\nov{t} \ ,
\end{displaymath}
a straightforward calculation with the
differentials $d_{H}$ and~$d_{V}$ yields:

\begin{proposition}
  \label{proposition:cone_Novikov_is_totrt}
  If $D$ consists of finitely presented $R_{0}$-modules, there is an
  isomorphism of $R\nov{t}$-module chain complexes
  $\totrt(E_{\bullet,\bullet}) \cong \cone(\nu) \tensor_{R[t,t^{-1}]}
  R\nov{t}$.\qed
\end{proposition}

\subsection*{From finite domination to trivial Novikov homology}

We are finally in a position to finish the proof of
our main result.

\begin{proof}[Proof of Theorem~\ref{thm:main}, ``only if'' part]
  Suppose that $R[t,t^{-1}]$ is strongly graded. Let $C$ be a bounded
  complex of finitely generated free $R[t,t^{-1}]$-modules; suppose
  that $C$ is $R_{0}$-finitely dominated. Then there is a bounded
  complex~$D$ of finitely generated projective $R_{0}$-modules
  together with a homotopy equivalence $\alpha \colon C \rTo D$ of
  $R_{0}$-module complexes. Let $\beta$~be a homotopy inverse
  of~$\alpha$. According to Corollary~\ref{findomimp} this data can be
  used to manufacture a homotopy equivalence $C \tensor_{R[t,t^{-1}]}
  R\nov{t} \simeq \cone(\nu) \tensor_{R[t,t^{-1}]} R\nov{t}$, where
  $\nu$ is a chain complex self-map of $D \tensor_{R_{0}} R[t,t^{-1}]$
  as in~\eqref{eq:map_nu}. We can use
  Proposition~\ref{proposition:cone_Novikov_is_totrt} to identify
  $\cone(\nu) \tensor_{R[t,t^{-1}]} R\nov{t}$ with $\totrt
  (E_{\bullet, \bullet})$, the right truncated totalisation of the
  double complex defined in~\eqref{eq:E}, as $D$ consists of finitely
  presented $R_{0}$\nobreakdash-modules. The vertical differential of
  this complex, defined in~\eqref{eq:dV}, is the mapping cone of
  $\alpha \beta \tensor \id$. As $\alpha \beta \simeq \id$ this means
  that the columns of~$E_{\bullet, \bullet}$ are acyclic, hence
  $\totrt(E_{\bullet, \bullet})$ is acyclic
  \cite[Proposition~1.2]{DCVNC}. This shows that $C
  \tensor_{R[t,t^{-1}]} R\nov{t}$, being homotopy equivalent to an
  acyclic complex, has trivial homology.

  To prove that $C \tensor_{R[t,t^{-1}]} R\nov{t^{-1}}$ has trivial
  homology too we cannot simply swap the roles of ``left'' and
  ``right'' as we did not analyse whether the {\it rows\/}
  of~$E_{\bullet, \bullet}$ are acyclic. Instead, we can quote what we
  proved so far, applied to the strongly $\bZ$-graded ring $\bar
  R[t,t^{-1}]$ with $n$th homogeneous component~$R_{-n}$ (which as a
  ring, neglecting the grading, coincides with~$R[t,t^{-1}]$). We then
  conclude that $C \tensor_{R[t,t^{-1}]} R\nov{t^{-1}} = C
  \tensor_{\bar R[t,t^{-1}]} \bar R\nov{t}$ has trivial homology as
  required.
\end{proof}

\raggedright


\begin{thebibliography}{Dad80}

\bibitem[Bou98]{MR1727844}
Nicolas Bourbaki.
\newblock {\em Algebra {I}. {C}hapters 1--3}.
\newblock Elements of Mathematics (Berlin). Springer-Verlag, Berlin, 1998.
\newblock Translated from the French, Reprint of the 1989 English translation.

\bibitem[Dad80]{GRD}
Everett~C. Dade.
\newblock Group-graded rings and modules.
\newblock {\em Math. Z.}, 174(3):241--262, 1980.

\bibitem[HK07]{MR2282258}
J.~A. Hillman and D.~H. Kochloukova.
\newblock Finiteness conditions and {${\rm PD}_r$}-group covers of {${\rm
  PD}_n$}-complexes.
\newblock {\em Math. Z.}, 256(1):45--56, 2007.

\bibitem[H{\"u}t11]{DCVNC}
Thomas H{\"u}ttemann.
\newblock Double complexes and vanishing of {N}ovikov cohomology.
\newblock {\em Serdica Math. J.}, 37(4):295--304 (2012), 2011.

\bibitem[H{\"u}t15]{TVB}
Thomas H{\"u}ttemann.
\newblock Vector bundles on the projective line and finite domination of chain
  complexes.
\newblock {\em Math. Proc. R. Ir. Acad.}, 115A(1):Art. 2, 12, 2015.

\bibitem[Ran85]{RATFO}
Andrew Ranicki.
\newblock The algebraic theory of finiteness obstruction.
\newblock {\em Math. Scand.}, 57(1):105--126, 1985.

\bibitem[Ran95]{RFDNR}
Andrew Ranicki.
\newblock Finite domination and {N}ovikov rings.
\newblock {\em Topology}, 34(3):619--632, 1995.

\end{thebibliography}
\end{document}